\documentclass{amsart}
\usepackage{amssymb}
\usepackage{amsfonts}

\newtheorem{theorem}{Theorem}
\theoremstyle{plain}

\newtheorem{corollary}{Corollary}

\newtheorem{definition}{Definition}

\newtheorem{proposition}{Proposition}
\newtheorem{remark}{Remark}

\numberwithin{equation}{section}
\input{tcilatex}

\begin{document}
\title[HADAMARD-TYPE\ INEQUALITIES]{ON $(h-m)-$CONVEXITY AND HADAMARD-TYPE\
INEQUALITIES}
\author{$^{\blacktriangledown }$M. Emin \"{O}zdemir}
\address{$^{\blacktriangledown }$Ataturk University, K. K. Education
Faculty, Department of Mathematics, 25640, Kampus, Erzurum, Turkey}
\email{emos@atauni.edu.tr}
\author{$^{\bigstar }$Ahmet Ocak Akdemir}
\address{$^{\bigstar }$A\u{g}r\i\ \.{I}brahim \c{C}e\c{c}en University,
Faculty of Science and Arts, Department of Mathematics, 04100, A\u{g}r\i ,
Turkey}
\email{ahmetakdemir@agri.edu.tr}
\author{$^{\blacktriangledown }$Erhan Set}
\address{$^{\blacktriangledown }$Department of Mathematics, Faculty of
Science and Arts, D\"{u}zce University, D\"{u}zce, Turkey}
\email{erhanset@yahoo.com}
\date{February 20, 2011}

\begin{abstract}
In this paper, a new class of convex functions as a generalization of
convexity which is called $(h-m)-$convex functions and some properties of
this class is given. We also prove some \ Hadamard's type inequalities.
\end{abstract}

\maketitle

\section{INTRODUCTION}

The concept of $m-$convexity has been introduced by Toader in \cite{TOA}, an
intermediate between the ordinary convexity and starshaped property, as
following:

\begin{definition}
The function $f:\left[ 0,b\right] \rightarrow 
\mathbb{R}
,$ $b>0,$ is said to be $m-$convex, where $m\in \left[ 0,1\right] ,$ if we
have%
\begin{equation*}
f\left( tx+m\left( 1-t\right) y\right) \leq tf\left( x\right) +m\left(
1-t\right) f\left( y\right)
\end{equation*}%
for all $x,y\in \left[ 0,b\right] $ and $t\in \left[ 0,1\right] .$ We say
that $f$ is $m-$concave if $-f$ is $m-$convex.
\end{definition}

Several papers have been written on $m-$convex functions and we refer the
papers \cite{MER}, \cite{TOA}, \cite{SS2}, \cite{TOA2}, \cite{BAK} and \cite%
{SS3}. In \cite{SS2}, Dragomir and Toader proved following inequality for $%
m- $convex functions.

\begin{theorem}
Let $f:\left[ 0,\infty \right) \rightarrow 
\mathbb{R}
$ be a $m-$convex function with $m\in (0,1].$ If $0\leq a<b<\infty $ and $%
f\in L_{1}\left[ a,b\right] ,$ then one has the inequality:%
\begin{equation}
\frac{1}{b-a}\int_{a}^{b}f\left( x\right) dx\leq \min \left\{ \frac{%
f(a)+mf\left( \frac{b}{m}\right) }{2},\frac{f(b)+mf\left( \frac{a}{m}\right) 
}{2}\right\} .  \label{m1}
\end{equation}
\end{theorem}

In \cite{SS3}, Dragomir established following inequalities of Hadamard-type
similar to above.

\begin{theorem}
Let $f:\left[ 0,\infty \right) \rightarrow 
\mathbb{R}
$ be a $m-$convex function with $m\in (0,1].$ If $0\leq a<b<\infty $ and $%
f\in L_{1}\left[ a,b\right] ,$ then one has the inequality:%
\begin{eqnarray}
f\left( \frac{a+b}{2}\right) &\leq &\frac{1}{b-a}\int_{a}^{b}\frac{%
f(x)+mf\left( \frac{x}{m}\right) }{2}dx  \label{m2} \\
&\leq &\frac{m+1}{4}\left[ \frac{f(a)+f\left( b\right) }{2}+m\frac{f(\frac{a%
}{m})+f\left( \frac{b}{m}\right) }{2}\right] .  \notag
\end{eqnarray}
\end{theorem}

\begin{theorem}
Let $f:\left[ 0,\infty \right) \rightarrow 
\mathbb{R}
$ be a $m-$convex function with $m\in (0,1].$ If $f\in L_{1}\left[ am,b%
\right] $ where $0\leq a<b<\infty ,$ then one has the inequality:%
\begin{equation}
\frac{1}{m+1}\left[ \int_{a}^{mb}f(x)dx+\frac{mb-a}{b-ma}\int_{ma}^{b}f(x)dx%
\right] \leq \left( mb-a\right) \frac{f(a)+f\left( b\right) }{2}.  \label{m3}
\end{equation}
\end{theorem}

In \cite{WW}, Breckner introduced a new class of convex functions, a
generalization of the ordinary convexity, is called $s-$convexity, as
following;

\begin{definition}
Let $s$ be a real number, $s\in (0,1].$ A function $f:[0,\infty )\rightarrow
\lbrack 0,\infty )$ is said to be $s-$convex (in the second sense), or that $%
f$ belongs to the class $K_{s}^{2}$, if 
\begin{equation*}
f\left( \alpha x+\left( 1-\alpha \right) y\right) \leq \alpha ^{s}f\left(
x\right) +\left( 1-\alpha \right) ^{s}f\left( y\right)
\end{equation*}%
for all $x,y\in \lbrack 0,\infty )$ and $\alpha \in \left[ 0,1\right] .$
\end{definition}

Some properties of $s-$convexity have been given in \cite{HU} and K\i
rmac\i\ \textit{et al.} proved some inequalities for $s-$convex functions in 
\cite{UK}. In \cite{SF}, Dragomir and Fitzpatrick established the following
Hadamard's type inequalities;

\begin{theorem}
Suppose that $f:\left[ 0,\infty \right) \rightarrow \left[ 0,\infty \right) $
is an $s-$convex function in the second sense, where $s\in \left( 0,1\right] 
$ and let $a,b\in \left[ 0,\infty \right) ,$ $a<b.$ If $f\in L_{1}\left[ 0,1%
\right] ,$ then the following inequalities hold: 
\begin{equation}
2^{s-1}f\left( \frac{a+b}{2}\right) \leq \frac{1}{b-a}\int_{a}^{b}f\left(
x\right) dx\leq \frac{f\left( a\right) +f\left( b\right) }{s+1}.  \label{s}
\end{equation}%
The constant $k=\frac{1}{s+1}$ is the best possible in the second inequality
in (\ref{s}). The above inequalities are sharp.
\end{theorem}

In \cite{GO}, Godunova and Levin introduced the following class of functions.

\begin{definition}
A function $f:I\subseteq 
\mathbb{R}
\rightarrow 
\mathbb{R}
$ is said to belong to the class of $Q(I)$ if it is nonnegative and, for all 
$x,y\in I$ and $\lambda \in (0,1)$ satisfies the inequality;%
\begin{equation*}
f(\lambda x+(1-\lambda )y)\leq \frac{f(x)}{\lambda }+\frac{f(y)}{1-\lambda }
\end{equation*}
\end{definition}

In \cite{SS1}, Dragomir \textit{et al}. defined following new class of
functions.

\begin{definition}
A function $f:I\subseteq 
\mathbb{R}
\rightarrow 
\mathbb{R}
$ is $P$ function or that $f$ belongs to the class of $P(I),$ if it is
nonnegative and for all $x,y\in I$ and $\lambda \in \lbrack 0,1],$ satisfies
the following inequality;

\begin{equation*}
f(\lambda x+(1-\lambda )y)\leq f(x)+f(y)
\end{equation*}
\end{definition}

In \cite{SS1}, Dragomir \textit{et al}. proved two inequalities of
Hadamard's type for class of Godunova-Levin functions and $P-$ functions.

\begin{theorem}
Let $f\in Q(I),a,b\in I,$ with $a<b$ and $f\in L_{1}[a,b].$ Then the
following inequality holds.
\end{theorem}

\begin{equation}
\ \ \ f\left( \frac{a+b}{2}\right) \leq \frac{4}{b-a}\int_{a}^{b}f(x)dx
\label{q}
\end{equation}

\begin{theorem}
Let $f\in P(I),a,b\in I,$ with $a<b$ and $f\in L_{1}[a,b].$ Then the
following inequality holds.%
\begin{equation}
f\left( \frac{a+b}{2}\right) \leq \frac{2}{b-a}\int_{a}^{b}f(x)dx\leq
2[f(a)+f(b)]  \label{p}
\end{equation}
\end{theorem}

On all of these, in \cite{SA}, Varo\v{s}anec defined $h-$convex functions
and gave some properties of this class of functions.

\begin{definition}
Let $h:J\subseteq 
\mathbb{R}
\rightarrow 
\mathbb{R}
$\ be a positive function . We say that $f:I\subseteq 
\mathbb{R}
\rightarrow 
\mathbb{R}
$ is $h-$convex function, or that $f$ belongs to the class $SX\left(
h,I\right) $, if $f$ is nonnegative and for all $x,y\in I$\ and $t\in \left(
0,1\right) $\ we have \ \ \ \ \ \ \ \ \ \ \ \ 
\begin{equation}
f\left( tx+\left( 1-t\right) y\right) \leq h\left( t\right) f\left( x\right)
+h\left( 1-t\right) f\left( y\right) .  \label{h}
\end{equation}
\end{definition}

If inequality (\ref{h}) is reversed, then $f$ is said to be $h-$concave,
i.e., $f\in SV\left( h,I\right) $. Obviously, if $h\left( t\right) =t$, then
all nonnegative convex functions belong to $SX\left( h,I\right) $\ and all
nonnegative concave functions belong to $SV(h,I);$ if $h(t)=\frac{1}{t},$
then $SX(h,I)=Q(I);$ if $h(t)=1,$ $SX(h,I)\supseteq P(I);$ and if $%
h(t)=t^{s},$ where $s\in \left( 0,1\right) ,$ then $SX(h,I)\supseteq
K_{s}^{2}.$

\begin{theorem}
(See \cite{SAR}, Theorem 6) Let $f\in SX(h,I),$ $a,b\in I,$ with $a<b$ and $%
f\in L_{1}\left( \left[ a,b\right] \right) .$ Then 
\begin{equation}
\frac{1}{2h\left( \frac{1}{2}\right) }f\left( \frac{a+b}{2}\right) \leq 
\frac{1}{\left( b-a\right) }\int_{a}^{b}f\left( x\right) dx\leq \left[
f\left( a\right) +f\left( b\right) \right] \int_{0}^{1}h\left( \alpha
\right) d\alpha  \label{h1}
\end{equation}
\end{theorem}

For some recent results for $h-$convex functions we refer the interest of
reader to the papers \cite{SAR}, \cite{SA2}, \cite{OZ1} and \cite{BU}.

The main aim of this paper is to give a new class of convex functions and to
give some properties of this functions, by using a similar way to proof of
properties of $h-$convexity (see \cite{SA}). Therefore, some inequalities of
Hadamard-type related to this new class of convex functions are given.

\section{MAIN\ RESULTS}

We will introduce a new class of convex functions in the following
definition.

\begin{definition}
Let $h:J\subseteq 
\mathbb{R}
\rightarrow 
\mathbb{R}
$ be a non-negative function. We say that $f:\left[ 0,b\right] \rightarrow 
\mathbb{R}
$ is a $\left( h-m\right) -$convex function , if $f$ is non-negative and for
all $x,y\in \left[ 0,b\right] $, $m\in \left[ 0,1\right] $ and $\alpha \in
\left( 0,1\right) ,$ we have 
\begin{equation}
f\left( \alpha x+m\left( 1-\alpha \right) y\right) \leq h\left( \alpha
\right) f\left( x\right) +mh\left( 1-\alpha \right) f\left( y\right) .
\label{1}
\end{equation}%
If the inequality (\ref{1}) is reversed, then $f$ is said to be $\left(
h-m\right) -$concave function on $\left[ 0,b\right] .$
\end{definition}

Obviously, if we choose $m=1$, then we have $h-$convex functions. If we
choose $h\left( \alpha \right) =\alpha $, then we obtain non-negative $m-$%
convex functions. If we choose $m=1$ and $h\left( \alpha \right) =\left\{
\alpha ,\text{ }1,\text{ }\frac{1}{\alpha },\text{ }\alpha ^{s}\right\} ,$
then we obtain the following classes of functions, non-negative convex
functions, $P-$functions, Godunova-Levin functions and $s-$convex functions
(in the second sense), respectively.

\begin{remark}
Let $h$ be a non-negative function such that 
\begin{equation*}
h\left( \alpha \right) \geq \alpha
\end{equation*}%
for all $\alpha \in \left( 0,1\right) .$ If $f$ is a non-negative $m-$convex
function on $\left[ 0,b\right] ,$ then for all $x,y\in \left[ 0,b\right] ,$ $%
m\in \left[ 0,1\right] $ and $\alpha \in \left( 0,1\right) ,$ we have 
\begin{equation*}
f\left( \alpha x+m\left( 1-\alpha \right) y\right) \leq \alpha f\left(
x\right) +m\left( 1-\alpha \right) f\left( y\right) \leq h\left( \alpha
\right) f\left( x\right) +mh\left( 1-\alpha \right) f\left( y\right) .
\end{equation*}%
This shows that $f$ is a $\left( h-m\right) -$convex function. By a similar
way, one can see that, if 
\begin{equation*}
h\left( \alpha \right) \leq \alpha
\end{equation*}%
for all $\alpha \in \left( 0,1\right) .$ Then, all non-negative $m-$concave
functions are $\left( h-m\right) -$concave function on $\left[ 0,b\right] .$
\end{remark}

\begin{proposition}
Let $h_{1},h_{2}$ be non-negative functions defined on $J\subseteq 
\mathbb{R}
$ such that%
\begin{equation*}
h_{2}\left( t\right) \leq h_{1}\left( t\right)
\end{equation*}%
for $t\in \left( 0,1\right) .$ If $f$ is $\left( h_{2}-m\right) -$convex,
then $f$ is $\left( h_{1}-m\right) -$convex.
\end{proposition}

\begin{proof}
If $f$ is $\left( h_{2}-m\right) -$convex, then for all $x,y\in \left[ 0,b%
\right] $ and $\alpha \in \left( 0,1\right) ,$ we can write%
\begin{eqnarray*}
f\left( \alpha x+m\left( 1-\alpha \right) y\right) &\leq &h_{2}\left( \alpha
\right) f\left( x\right) +mh_{2}\left( 1-\alpha \right) f\left( y\right) \\
&\leq &h_{1}\left( \alpha \right) f\left( x\right) +mh_{1}\left( 1-\alpha
\right) f\left( y\right) .
\end{eqnarray*}%
Which completes the proof of $\left( h_{1}-m\right) -$convexity of $f$.
\end{proof}

\begin{proposition}
If $f,g$ are $\left( h-m\right) -$convex functions and $\lambda >0,$ then $%
f+g$ and $\lambda f$ \ are $\left( h-m\right) -$convex functions.
\end{proposition}

\begin{proof}
By using definition of $\left( h-m\right) -$convex functions, we can write%
\begin{equation}
f\left( \alpha x+m\left( 1-\alpha \right) y\right) \leq h\left( \alpha
\right) f\left( x\right) +mh\left( 1-\alpha \right) f\left( y\right)
\label{2}
\end{equation}%
and%
\begin{equation}
g\left( \alpha x+m\left( 1-\alpha \right) y\right) \leq h\left( \alpha
\right) g\left( x\right) +mh\left( 1-\alpha \right) g\left( y\right)
\label{3}
\end{equation}%
for all $x,y\in \left[ 0,b\right] ,$ $m\in \left[ 0,1\right] $ and $\alpha
\in \left( 0,1\right) .$ If we add (\ref{2}) and (\ref{3}), we get%
\begin{equation*}
\left( f+g\right) \left( \alpha x+m\left( 1-\alpha \right) y\right) \leq
h\left( \alpha \right) \left( f+g\right) \left( x\right) +mh\left( 1-\alpha
\right) \left( f+g\right) \left( y\right) .
\end{equation*}%
This shows that $f+g$ is $\left( h-m\right) -$convex function. Therefore, to
prove $\left( h-m\right) -$convexity of $\lambda f,$ from the definition, we
have%
\begin{equation*}
\lambda f\left( \alpha x+m\left( 1-\alpha \right) y\right) \leq h\left(
\alpha \right) \lambda f\left( x\right) +mh\left( 1-\alpha \right) \lambda
f\left( y\right) .
\end{equation*}%
This completes the proof.
\end{proof}

The following inequality of Hadamard-type for $\left( h-m\right) -$convex
functions holds.

\begin{theorem}
Let $f:[0,\infty )\rightarrow 
\mathbb{R}
$ be a $\left( h-m\right) -$convex function with $m\in (0,1],$ $t\in \left[
0,1\right] .$ If $0\leq a<b<\infty $ and $f\in L_{1}\left[ a,b\right] ,$
then the following inequality holds;%
\begin{eqnarray}
\frac{1}{b-a}\dint\limits_{a}^{b}f(x)dx &\leq &\min \left\{ f\left( a\right)
\dint\limits_{0}^{1}h\left( t\right) dt+mf\left( \frac{b}{m}\right)
\dint\limits_{0}^{1}h\left( 1-t\right) dt,\right.  \label{4} \\
&&\left. f\left( b\right) \dint\limits_{0}^{1}h\left( t\right) dt+mf\left( 
\frac{a}{m}\right) \dint\limits_{0}^{1}h\left( 1-t\right) dt\right\} . 
\notag
\end{eqnarray}
\end{theorem}

\begin{proof}
From the definition of $\left( h-m\right) -$convex functions, we can write%
\begin{equation*}
f\left( tx+m\left( 1-t\right) y\right) \leq h\left( t\right) f\left(
x\right) +mh\left( 1-t\right) f\left( y\right)
\end{equation*}%
for all $x,y\geq 0.$ It follows that; for all $t\in \left[ 0,1\right] ,$%
\begin{equation*}
f\left( ta+\left( 1-t\right) b\right) \leq h\left( t\right) f\left( a\right)
+mh\left( 1-t\right) f\left( \frac{b}{m}\right)
\end{equation*}%
and%
\begin{equation*}
f\left( tb+\left( 1-t\right) a\right) \leq h\left( t\right) f\left( b\right)
+mh\left( 1-t\right) f\left( \frac{a}{m}\right) .
\end{equation*}%
Integrating these inequalities on $\left[ 0,1\right] ,$ with respect to $t,$
we obtain%
\begin{equation*}
\dint\limits_{0}^{1}f\left( ta+\left( 1-t\right) b\right) dt\leq f\left(
a\right) \dint\limits_{0}^{1}h\left( t\right) dt+mf\left( \frac{b}{m}\right)
\dint\limits_{0}^{1}h\left( 1-t\right) dt
\end{equation*}%
and%
\begin{equation*}
\dint\limits_{0}^{1}f\left( tb+\left( 1-t\right) a\right) dt\leq f\left(
b\right) \dint\limits_{0}^{1}h\left( t\right) dt+mf\left( \frac{a}{m}\right)
\dint\limits_{0}^{1}h\left( 1-t\right) dt.
\end{equation*}%
It is easy to see that;%
\begin{equation*}
\dint\limits_{0}^{1}f\left( ta+\left( 1-t\right) b\right)
dt=\dint\limits_{0}^{1}f\left( tb+\left( 1-t\right) a\right) dt=\frac{1}{b-a}%
\dint\limits_{a}^{b}f(x)dx.
\end{equation*}%
Using this equality, we obtain the required result.
\end{proof}

\begin{corollary}
If we choose $h(t)=1$ in (\ref{4})$,$ we obtain the following inequality;%
\begin{equation*}
\frac{1}{b-a}\dint\limits_{a}^{b}f(x)dx\leq \min \left\{ f\left( a\right)
+mf\left( \frac{b}{m}\right) ,\text{ }f\left( b\right) +mf\left( \frac{a}{m}%
\right) \right\} .
\end{equation*}
\end{corollary}

\begin{corollary}
If we choose $m=1$ in (\ref{4})$,$ we obtain the following inequality;%
\begin{eqnarray*}
\frac{1}{b-a}\dint\limits_{a}^{b}f(x)dx &\leq &\min \left\{ f\left( a\right)
\dint\limits_{0}^{1}h\left( t\right) dt+f\left( b\right)
\dint\limits_{0}^{1}h\left( 1-t\right) dt,\right. \\
&&\left. f\left( b\right) \dint\limits_{0}^{1}h\left( t\right) dt+f\left(
a\right) \dint\limits_{0}^{1}h\left( 1-t\right) dt\right\} .
\end{eqnarray*}
\end{corollary}

\begin{remark}
If we choose $h(t)=t$ in (\ref{4})$,$ we obtain the inequality (\ref{m1}).
\end{remark}

\begin{remark}
If we choose $m=1$ and $h(t)=t$ in (\ref{4})$,$ we obtain the right hand
side of the Hadamard's inequality. If we choose $m=1$ and $h(t)=1$ in (\ref%
{4})$,$ we obtain the right hand side of the inequality (\ref{p}). If we
choose $m=1$ and $h(t)=t^{s}$ in (\ref{4})$,$ we obtain the right hand side
of the inequality (\ref{s}).
\end{remark}

Another result of Hadamard-type for $(h-m)-$convex functions is emboided in
the following theorem.

\begin{theorem}
Let $f:[0,\infty )\rightarrow 
\mathbb{R}
$ be a $\left( h-m\right) -$convex function with $m\in (0,1],$ $t\in \left[
0,1\right] .$ If $0\leq a<b<\infty $ and $f\in L_{1}\left[ a,b\right] ,$
then the following inequality holds;%
\begin{eqnarray}
f\left( \frac{a+b}{2}\right) &\leq &\frac{h\left( \frac{1}{2}\right) }{b-a}%
\dint\limits_{a}^{b}\left[ f(x)+mf\left( \frac{x}{m}\right) \right] dx
\label{5} \\
&\leq &h\left( \frac{1}{2}\right) \left[ \frac{f(a)+mf\left( \frac{b}{m}%
\right) +mf\left( \frac{a}{m}\right) +m^{2}f\left( \frac{b}{m^{2}}\right) }{2%
}\right]  \notag
\end{eqnarray}
\end{theorem}

\begin{proof}
For $x,y\in \lbrack 0,\infty )$ and $\alpha =\frac{1}{2}$, we can write
definition of $\left( h-m\right) -$convex function as following;%
\begin{equation*}
f\left( \frac{x+y}{2}\right) \leq h\left( \frac{1}{2}\right) f\left(
x\right) +mh\left( \frac{1}{2}\right) f\left( \frac{y}{m}\right)
\end{equation*}%
If we choose $x=ta+\left( 1-t\right) b$ and $y=tb+\left( 1-t\right) a,$ we
get%
\begin{equation*}
f\left( \frac{a+b}{2}\right) \leq h\left( \frac{1}{2}\right) f\left(
ta+\left( 1-t\right) b\right) +mh\left( \frac{1}{2}\right) f\left( \left(
1-t\right) \frac{a}{m}+t\frac{b}{m}\right)
\end{equation*}%
for all $t\in \left[ 0,1\right] .$ By integrating the result on $\left[ 0,1%
\right] $ with respect to $t,$ we have%
\begin{equation}
f\left( \frac{a+b}{2}\right) \leq h\left( \frac{1}{2}\right)
\dint\limits_{0}^{1}f\left( ta+\left( 1-t\right) b\right) dt+mh\left( \frac{1%
}{2}\right) \dint\limits_{0}^{1}f\left( \left( 1-t\right) \frac{a}{m}+t\frac{%
b}{m}\right) dt.  \label{6}
\end{equation}%
By the facts that%
\begin{equation*}
\dint\limits_{0}^{1}f\left( ta+\left( 1-t\right) b\right) dt=\frac{1}{b-a}%
\dint\limits_{a}^{b}f(x)dx
\end{equation*}%
and%
\begin{equation*}
\dint\limits_{0}^{1}f\left( \left( 1-t\right) \frac{a}{m}+t\frac{b}{m}%
\right) dt=\frac{m}{b-a}\dint\limits_{\frac{a}{m}}^{\frac{b}{m}}f(x)dx=\frac{%
1}{b-a}\dint\limits_{a}^{b}f\left( \frac{x}{m}\right) dx.
\end{equation*}%
Using these equalities in (\ref{6}), we obtain the first inequality of (\ref%
{5}). By the $\left( h-m\right) -$convexity of $f$, we can write%
\begin{eqnarray}
&&h\left( \frac{1}{2}\right) \left[ f\left( ta+\left( 1-t\right) b\right)
+mf\left( \left( 1-t\right) \frac{a}{m}+t\frac{b}{m}\right) \right]
\label{7} \\
&\leq &h\left( \frac{1}{2}\right) \left[ tf(a)+m\left( 1-t\right) f\left( 
\frac{b}{m}\right) +m\left( 1-t\right) f\left( \frac{a}{m}\right)
+m^{2}tf\left( \frac{b}{m^{2}}\right) \right] .  \notag
\end{eqnarray}%
Integrating the inequality (\ref{7}) on on $\left[ 0,1\right] $ with respect
to $t,$ we have%
\begin{eqnarray*}
&&\frac{h\left( \frac{1}{2}\right) }{b-a}\left[ \dint\limits_{a}^{b}f(x)dx+m%
\dint\limits_{a}^{b}f\left( \frac{x}{m}\right) dx\right] \\
&\leq &h\left( \frac{1}{2}\right) \left[ \frac{f(a)+mf\left( \frac{b}{m}%
\right) +mf\left( \frac{a}{m}\right) +m^{2}f\left( \frac{b}{m^{2}}\right) }{2%
}\right]
\end{eqnarray*}%
which completes the proof.
\end{proof}

\begin{corollary}
If we choose $h(t)=1$ in (\ref{5})$,$ we obtain the following inequality;%
\begin{eqnarray*}
f\left( \frac{a+b}{2}\right) &\leq &\frac{1}{b-a}\dint\limits_{a}^{b}\left[
f(x)+mf\left( \frac{x}{m}\right) \right] dx \\
&\leq &\left[ \frac{f(a)+mf\left( \frac{b}{m}\right) +mf\left( \frac{a}{m}%
\right) +m^{2}f\left( \frac{b}{m^{2}}\right) }{2}\right] .
\end{eqnarray*}
\end{corollary}

\begin{corollary}
If we choose $m=1$ and $h(t)=t^{s}$ in (\ref{5})$,$ we obtain the following
inequality;%
\begin{equation*}
2^{s-1}f\left( \frac{a+b}{2}\right) \leq \frac{1}{b-a}\dint%
\limits_{a}^{b}f(x)dx\leq \frac{f(a)+f\left( b\right) }{2}
\end{equation*}%
which is similar to (\ref{s}).
\end{corollary}

\begin{remark}
If we choose $m=1$ in (\ref{4})$,$ we obtain the right hand side of the
inequality (\ref{h1}).
\end{remark}

\begin{remark}
If we choose $m=1$ and $h(t)=t$ in (\ref{5})$,$ we obtain the Hadamard's
inequality.
\end{remark}

\begin{remark}
If we choose $h(t)=t$ in (\ref{5})$,$ we obtain the inequality (\ref{m2}).
\end{remark}

The following inequality also holds for $\left( h-m\right) -$convex
functions.

\begin{theorem}
Let $f:[0,\infty )\rightarrow 
\mathbb{R}
$ be a $\left( h-m\right) -$convex function with $m\in (0,1],$ $t\in \left[
0,1\right] .$ If $0\leq a<b<\infty $ and $f\in L_{1}\left[ ma,b\right] ,$
then the following inequality holds;%
\begin{eqnarray}
&&\frac{1}{m+1}\left[ \frac{1}{mb-a}\dint\limits_{a}^{mb}f(x)dx+\frac{1}{b-ma%
}\dint\limits_{ma}^{b}f(x)dx\right]  \label{8} \\
&\leq &\frac{f\left( a\right) +f\left( b\right) }{2}\left[
\dint\limits_{0}^{1}h\left( t\right) dt+\dint\limits_{0}^{1}h\left(
1-t\right) dt\right] .  \notag
\end{eqnarray}
\end{theorem}

\begin{proof}
From definition of $\left( h-m\right) -$convex functions, we can write%
\begin{equation*}
f\left( ta+m\left( 1-t\right) b\right) \leq h\left( t\right) f\left(
a\right) +mh\left( 1-t\right) f\left( b\right)
\end{equation*}%
\begin{equation*}
f\left( \left( 1-t\right) a+mtb\right) \leq h\left( 1-t\right) f\left(
a\right) +mh\left( t\right) f\left( b\right)
\end{equation*}%
\begin{equation*}
f\left( tb+m\left( 1-t\right) a\right) \leq h\left( t\right) f\left(
b\right) +mh\left( 1-t\right) f\left( a\right)
\end{equation*}%
and%
\begin{equation*}
f\left( \left( 1-t\right) b+mta\right) \leq h\left( 1-t\right) f\left(
b\right) +mh\left( t\right) f\left( a\right)
\end{equation*}%
for all $t\in \left[ 0,1\right] .$ By summing these inequalities and
integrating on $\left[ 0,1\right] $ with respect to $t,$ we obtain%
\begin{eqnarray}
&&\dint\limits_{0}^{1}f\left( ta+m\left( 1-t\right) b\right)
dt+\dint\limits_{0}^{1}f\left( \left( 1-t\right) a+mtb\right) dt  \label{9}
\\
&&+\dint\limits_{0}^{1}f\left( tb+m\left( 1-t\right) a\right)
dt+\dint\limits_{0}^{1}f\left( \left( 1-t\right) b+mta\right) dt  \notag \\
&\leq &\left( f\left( a\right) +f\left( b\right) \right) \left( m+1\right) 
\left[ \dint\limits_{0}^{1}h\left( t\right) dt+\dint\limits_{0}^{1}h\left(
1-t\right) dt\right] .  \notag
\end{eqnarray}%
It is easy to show that%
\begin{equation*}
\dint\limits_{0}^{1}f\left( ta+m\left( 1-t\right) b\right)
dt=\dint\limits_{0}^{1}f\left( \left( 1-t\right) a+mtb\right) dt=\frac{1}{%
mb-a}\dint\limits_{a}^{mb}f(x)dx
\end{equation*}%
and%
\begin{equation*}
\dint\limits_{0}^{1}f\left( tb+m\left( 1-t\right) a\right)
dt=\dint\limits_{0}^{1}f\left( \left( 1-t\right) b+mta\right) dt=\frac{1}{%
b-ma}\dint\limits_{ma}^{b}f(x)dx.
\end{equation*}%
By using these equalities in (\ref{9}), we get the desired result.
\end{proof}

\begin{corollary}
If we choose $h(t)=1$ in (\ref{8})$,$ we obtain the following inequality;%
\begin{equation*}
\frac{1}{m+1}\left[ \frac{1}{mb-a}\dint\limits_{a}^{mb}f(x)dx+\frac{1}{b-ma}%
\dint\limits_{ma}^{b}f(x)dx\right] \leq f\left( a\right) +f\left( b\right) .
\end{equation*}
\end{corollary}

\begin{remark}
If we choose $m=1$ and $h(t)=t$ in (\ref{8})$,$ we obtain the right hand
side of the Hadamard's inequality. If we choose $m=1$ and $h(t)=1$ in (\ref%
{8})$,$ we obtain the right hand side of the inequality (\ref{p}). If we
choose $m=1$ and $h(t)=t^{s}$ in (\ref{8})$,$ we obtain the right hand side
of the inequality (\ref{s}).
\end{remark}

\begin{remark}
If we choose $h(t)=t$ in (\ref{8})$,$ we obtain the inequality (\ref{m3}).
\end{remark}

\end{document}